\documentclass[12pt,oneside]{article}
\usepackage{amsmath,amssymb,amsfonts,amsthm}
\newcommand{\T}{{\cal T}}

\newcommand{\set}[1]{\left\{#1\right\}}

\newcommand{\p}{\pi^{-1}(TM)}
\newcommand {\cp}{\mathfrak{X}(\pi (M))}

\def\Section#1{\vspace{30truept}\addtocounter{section}{1}\setcounter{thm}{0}
\setcounter{equation}{0}{\noindent\Large\bf
    \arabic{section}.~~#1}\par \vspace{12pt}}

\newtheorem{thm}{Theorem}[section]

\newtheorem{lem}[thm]{Lemma}
\newtheorem{prop}[thm]{Proposition}
\newtheorem{defn}[thm]{Definition}

\newtheorem{rem}[thm]{Remark}

\numberwithin{equation}{section}

\begin{document}
\title{{\textbf{On Concircularly Recurrent Finsler  Manifolds}}}   
\author{{\bf Nabil L. Youssef$^{\,1, 2}$ and A. Soleiman$^{3}$}}
\date{}

\maketitle                     
\vspace{-1.15cm}
\begin{center}
{$^{1}$Department of Mathematics, Faculty of Science,\\ Cairo
University, Giza, Egypt}
\end{center}
\vspace{-0.6cm}
\begin{center}
{$^{2}$ Center of Theoretical Physics (CTP) \\at the British
University in Egypt (BUE)}
\end{center}
\vspace{-0.6cm}
\begin{center}
{$^{3}$Department of Mathematics, Faculty of Science,\\ Benha
University, Benha, Egypt}
\end{center}
\vspace{-0.6cm}
\begin{center}
E-mails: nlyoussef@sci.cu.edu.eg, nlyoussef2003@yahoo.fr\\
  {\hspace{2.3cm}}  amr.hassan@fsci.bu.edu.eg, amrsoleiman@yahoo.com
\end{center}
\smallskip

\vspace{1cm} \maketitle
\smallskip


\noindent{\bf Abstract.}
Two special Finsler spaces have been introduced and investigated, namely $R^h$-recurrent Finsler space and consircularly recurrent Finsler space. The defining properties of these spaces are formulated in terms of the first curvature tensor of Cartan connection.
The following three results constitute the main object of the present paper:
1. A concircularly flat Finsler manifold is necessarily of constant curvature (Theorem A);
2. Every $R^h$-recurrent Finsler manifold is concirculaly recurrent with the same recurrence form (Theorem B);
3. Every horizontally integrable concircularly recurrent Finsler manifold is $R^h$-recurrent with the same recurrence form (Theorem C).
The whole work is formulated in a coordinate-free
form.


\newpage


\centerline{\Large\bf{Introduction}}\vspace{12pt}
\par
In this paper, we present an intrinsic investigation
of concircularly recurrent Finsler manifolds. 
The paper is organized in the following manner. 

In section 1, following the introduction, we give a brief account of the basic
concepts, definitions and results that will be needed in the sequel. 

In section 2, an important tensor field associated to a Finsler manifold, called the concircular curvature tensor, is defined. A necessary and sufficient condition for the vanishing of the concircular curvature tensor is found (Proposition 2.4). We also prove that a concircularly flat Finsler manifold is necessairly of constant curvature (Theorem A).

In section 3, two special Finsler spaces have been introduced and investigated, namely $R^h$-recurrent Finsler space and consircularly recurrent Finsler space. The defining properties of these spaces are formulated in terms of the first curvature tensor of Cartan connection.
Then, we prove that every $R^h$-recurrent Finsler manifold is concircularly recurrent with the same recurrence form (Theorems B).
The converse of the above result is not true in general.
However, it has been recently proved to be true in Riemannian geometry \cite{Olszak}.
For the converse of Theorem B to be true in the Finslerian context, an additional
condition is needed, namely the horizontal integrability condition.
We thus prove that every horizontally integrable concircularly recurrent Finsler manifold is $R^h$-recurrent with the same recurrence form (Theorem~C). This is the third and most important result of the paper.

Finally, it should be pointed out that the present work is
formulated in a coordinate-free form. 


\Section{Notation and Preliminaries}

In this section, we give a brief account of the basic concepts
 of the pullback approach to intrinsic Finsler geometry necessary for this work. For more
 details, we refer to \cite{r58}, \cite{r86} and \cite{r94}. We
 shall use the same notations of \cite{r86}.

 In what follows, we denote by $\pi: \T M\longrightarrow M$ the subbundle of nonzero vectors
tangent to $M$, $\mathfrak{F}(TM)$ the algebra of $C^\infty$ functions on $TM$, $\cp$ the $\mathfrak{F}(TM)$-module of differentiable sections of the pullback bundle $\pi^{-1}(T M)$.
The elements of $\mathfrak{X}(\pi (M))$ will be called $\pi$-vector
fields and will be denoted by barred letters $\overline{X} $. The
tensor fields on $\pi^{-1}(TM)$ will be called $\pi$-tensor fields.
The fundamental $\pi$-vector field is the $\pi$-vector field
$\overline{\eta}$ defined by $\overline{\eta}(u)=(u,u)$ for all
$u\in TM$.
\par
We have the following short exact sequence of vector bundles
$$0\longrightarrow
 \pi^{-1}(TM)\stackrel{\gamma}\longrightarrow T(\T M)\stackrel{\rho}\longrightarrow
\pi^{-1}(TM)\longrightarrow 0 ,\vspace{-0.1cm}$$ with the well known
definitions of  the bundle morphisms $\rho$ and $\gamma$. The vector
space $V_u (\T M)= \{ X \in T_u (\T M) : d\pi(X)=0 \}$  is the vertical space to $M$ at $u$.
\par
Let $D$ be  a linear connection on the pullback bundle $\pi^{-1}(TM)$.
 We associate with $D$ the map \vspace{-0.1cm} $K:T \T M\longrightarrow
\pi^{-1}(TM):X\longmapsto D_X \overline{\eta} ,$ called the
connection map of $D$.  The vector space $H_u (\T M)= \{ X \in T_u
(\T M) : K(X)=0 \}$ is called the horizontal space to $M$ at $u$ .
   The connection $D$ is said to be regular if
$$ T_u (\T M)=V_u (\T M)\oplus H_u (\T M) \,\,\,  \forall \, u\in \T M.$$

If $M$ is endowed with a regular connection, then the vector bundle
   maps $
 \gamma,\, \rho |_{H(\T M)}$ and $K |_{V(\T M)}$
 are vector bundle isomorphisms. The map
 $\beta:=(\rho |_{H(\T M)})^{-1}$
 will be called the horizontal map of the connection
$D$.
\par
 The horizontal ((h)h-) and
mixed ((h)hv-) torsion tensors of $D$, denoted by $Q $ and $ T $
respectively, are defined by \vspace{-0.2cm}
$$Q (\overline{X},\overline{Y})=\textbf{T}(\beta \overline{X}\beta \overline{Y}),
\, \,\,\, T(\overline{X},\overline{Y})=\textbf{T}(\gamma
\overline{X},\beta \overline{Y}) \quad \forall \,
\overline{X},\overline{Y}\in\mathfrak{X} (\pi (M)),\vspace{-0.2cm}$$
where $\textbf{T}$ is the (classical) torsion tensor field
associated with $D$.
\par
The horizontal (h-), mixed (hv-) and vertical (v-) curvature tensors
of $D$, denoted by $R$, $P$ and $S$
respectively, are defined by
$$R(\overline{X},\overline{Y})\overline{Z}=\textbf{K}(\beta
\overline{X}\beta \overline{Y})\overline{Z},\quad
 {P}(\overline{X},\overline{Y})\overline{Z}=\textbf{K}(\beta
\overline{X},\gamma \overline{Y})\overline{Z},\quad
 {S}(\overline{X},\overline{Y})\overline{Z}=\textbf{K}(\gamma
\overline{X},\gamma \overline{Y})\overline{Z}, $$
 where $\textbf{K}$
is the (classical) curvature tensor field associated with $D$.
\par
The contracted curvature tensors of $D$, denoted by $\widehat{
{R}}$, $\widehat{ {P}}$ and $\widehat{ {S}}$ respectively, known
also as the
 (v)h-, (v)hv- and (v)v-torsion tensors, are defined by
$$\widehat{ {R}}(\overline{X},\overline{Y})={ {R}}(\overline{X},\overline{Y})\overline{\eta},\quad
\widehat{ {P}}(\overline{X},\overline{Y})={
{P}}(\overline{X},\overline{Y})\overline{\eta},\quad \widehat{
{S}}(\overline{X},\overline{Y})={
{S}}(\overline{X},\overline{Y})\overline{\eta}.$$
\par
If $M$ is endowed with a metric $g$ on $\p$, we write
\begin{equation}\label{cur.g}
   \textbf{ R}(\overline{X},\overline{Y},\overline{Z}, \overline{W}):
=g( {R}(\overline{X},\overline{Y})\overline{Z}, \overline{W}),\,
\cdots, \, \textbf{S}(\overline{X},\overline{Y},\overline{Z},
\overline{W}): =g( {S}(\overline{X},\overline{Y})\overline{Z},
\overline{W})
\end{equation}
\par
The following result is of extreme importance. \vspace{-0.1cm}
\begin{thm} {\em\cite{r94}} \label{th.1} Let $(M,L)$ be a Finsler
manifold and  $g$ the Finsler metric defined by $L$. There exists a
unique regular connection $\nabla$ on $\pi^{-1}(TM)$ such
that\vspace{-0.2cm}
\begin{description}
  \item[(a)]  $\nabla$ is  metric\,{\em:} $\nabla g=0$,

  \item[(b)] The (h)h-torsion of $\nabla$ vanishes\,{\em:} $Q=0
  $,
  \item[(c)] The (h)hv-torsion $T$ of $\nabla$\, satisfies\,\emph{:}
   $g(T(\overline{X},\overline{Y}), \overline{Z})=g(T(\overline{X},\overline{Z}),\overline{Y})$.
\end{description}
\par
 Such a connection is called the Cartan
connection associated with  the Finsler manifold $(M,L)$.
\end{thm}

On a Finsler manifold there are other important linear connections. However, the only linear connection we treat in this paper is the Cartan connection.
For a Finsler manifold  $(M,L)$, we
define the following geometric objects:
\begin{eqnarray*}
\ell &:=&L^{-1}i_{\overline{\eta}}\:g,\\
\hbar &:=& g-\ell \otimes \ell: \text{the angular metric tensor},\\
\phi &:& \text{the vector $\pi$-form associated with $\hbar$}; \,\,\, i_{\phi(\overline{X})}\,g:=i_{\overline{X}}\:\hbar\\
\stackrel{h}\nabla&:&\text{the $h$-covariant derivatives associated
with the Cartan connection},\\
\stackrel{v}\nabla &:&\text{the $v$-covariant derivatives associated
with the Cartan connection},\\
T&:& \text{the  Cartan tensor}; \,\,\, T(\overline{X},\overline{Y},\overline{Z}):=g(T(\overline{X},\overline{Y}),\overline{Z})\\
H&:=& i_{\overline{\eta}}\:\widehat{R}:  \text{the deviation tensor},\\
 Ric&:& \text{the  horizontal  Ricci tensor of Cartan connection
 },\\
 r &:& \text{the horizontal scalar curvature of Cartan connection }.
\end{eqnarray*}


\Section{Concircularly Flat Finsler Manifold}

\begin{defn}\label{def.3} \emph{\cite{r86}, \cite{r65} }A Finsler manifold $(M,L)$ of dimension $n\geq3$
is said to be $h$-isotropic if there exists a scalar function
$k_{o}\neq0$ such that the horizontal curvature tensor ${R}$ has the
form:
$$ {R}=k_{o}\, G,$$
where $G$ is the $\pi$-tensor field defined by
\begin{equation}\label{G}
G(\overline{X},\overline{Y})\overline{Z}:=g(\overline{X},\overline{Z})
\overline{Y}-g(\overline{Y},\overline{Z})\overline{X}.
\end{equation}
\end{defn}

\begin{defn}\label{sca.}{\em{\cite{r86}}, \cite{nr1}}  A Finsler manifold $(M,L)$ of dimension $n\geq 3$ is
said to be  of scalar curvature  if the deviation tensor
$H:=i_{\overline{\eta}}\:\widehat{R}$ satisfies
 $$H=\varepsilon L^{2} \phi, $$
  where $\varepsilon$ is a  scalar function on $T
  M$, positively homogenous of degree zero in $y$.
\par In particular, if the scalar function $\varepsilon$ is constant, then $(M,L)$ is said to be of constant curvature.
\end{defn}

Let us now introduce the notion of concircular curvture.
\vspace{-5pt}
\begin{defn}\label{def.1}
Let $(M,L)$ be a Finsler manifold of dimension $n\geq3$. The
$\pi$-tensor field $ {C}$
 defined by
 $$ {C}:= {R}-\frac{r}{n(n-1)}\, {G}$$
will be called the concircular curvature tensor, ${G}$ being the $\pi$-tensor field defined by (\ref{G}).
\par If the concircular curvature tensor ${C}$ vanishes, then $(M,L)$ is
said to be concircularly flat.
\end{defn}

It should be noted that the concircular curvature tensor in \emph{Riemannian geometry} has been thoroughly investigated by many authors. The above definition is a generalization to Finsler geometry of that tensor field.

\begin{prop}\label{th.1} A Finsler manifold $(M,L)$ is  concircularly flat if,
and only if, $(M,L)$ is $h$-isotropic.
\end{prop}

\begin{proof} It is clear that if $(M,L)$ is concircularly flat,
then it is $h$-isotropic (with $k_{o}=\frac{r}{n(n-1)}$ in Definition \ref{def.3}).
\par Conversely, suppose that $(M,L)$ be $h$-isotropic. Then,  by
Definition \ref{def.3}, we have
\begin{equation}\label{eq.a}
{R}(\overline{X},\overline{Y})\overline{Z}=k_{o} \set{g(\overline{X},\overline{Z})
\overline{Y}-g(\overline{Y},\overline{Z})\overline{X}}
\end{equation}
Taking the   trace with respect to $\overline{Y}$ of the above
relation, we get
$$Ric(\overline{X},\overline{Z})=k_{o}\set{ng(\overline{X},\overline{Z})
-g(\overline{X},\overline{Z})}.$$
This equation, again, by taking the
  trace with respect to the pair of arguments
$\overline{X}$ and $\overline{Z}$, reduces to
 $$ k_{o}=\frac{r}{n(n-1)}.$$
From which, taking into account (\ref{eq.a}) and Definition
\ref{def.1}, $(M,L)$ is therefore concirculary flat.
\end{proof}

\vspace{5pt}
The following theorem is one of the main results of the present paper.

\vspace{5pt}

 \noindent\textbf{Theorem A.} \label{th.3}  \emph{A concircularly flat Finsler manifold is necessarily of constant\linebreak
curvature.}

\vspace{5pt}

To prove this theorem we need the following three lemmas.
\begin{lem}\label{ell}
For  a Finsler manifold $(M,L)$, we have:
\begin{description}
 \item[{\textbf{(a)}}]
$\stackrel{h}{\nabla}L=0,\quad\stackrel{v}{\nabla}L=\ell.$

 \item[{\textbf{(b)}}]
$\stackrel{h}{\nabla}\ell=0,
\quad\stackrel{v}{\nabla}\ell=L^{-1}\hbar.$

 \item[{\textbf{(c)}}]
$i_{\overline{\eta}}\,\ell=L, \quad i_{\overline{\eta}}\hbar=0.$

\item[{\textbf{(d)}}]
$\phi=I-L^{-1}\ell\otimes \overline{\eta}.$
\end{description}
\end{lem}

\begin{lem}\label{lem.2} For a concircularly flat Finsler manifold $(M,L)$,
we have
$$\mathfrak{S}_{\overline{X},\overline{Y},\overline{Z}} {R}(\overline{X},\overline{Y})\overline{Z}=0.$$
\end{lem}
\begin{proof} Let $(M,L)$ be concircularly flat. Then, by Definition \ref{def.1} and the fact that $i_{\overline{\eta}}\,g=L\ell$, we have

\begin{equation}\label{eq.1}
    \widehat{ {R}}(\overline{X},\overline{Y})=kL(\ell(\overline{X})\overline{Y}-\ell(\overline{Y})\overline{X}),
 \end{equation}
where $k(x,y):=\frac{r}{n(n-1)}$, necessarily homogenous of degree $0$ in $y$.\\
From (\ref{eq.1}), taking into account the fact that the
$(h)hv$-torsion $T$ is symmetric, we obtain
\begin{eqnarray}
   \mathfrak{S}_{\overline{X},\overline{Y},\overline{Z}}T(\widehat{ {R}}(\overline{X},\overline{Y}),\overline{Z})
      &=&kL\set{
   \ell(\overline{X})T(\overline{Y},\overline{Z})-
    \ell(\overline{Y})T(\overline{X},\overline{Z})}\nonumber\\
    &&+kL\set{
   \ell(\overline{Y})T(\overline{Z},\overline{X})-
    \ell(\overline{Z})T(\overline{Y},\overline{X})}\nonumber\\
    &&+kL\set{
   \ell(\overline{Z})T(\overline{X},\overline{Y})-
    \ell(\overline{X})T(\overline{Z},\overline{Y})}\nonumber\\
    &=&0. \label{eq.3}
\end{eqnarray}
On the other hand, we have \cite{r96}
\begin{equation}\label{eq.2}
   \mathfrak{S}_{\overline{X},\overline{Y},\overline{Z}} {R}(\overline{X},\overline{Y})\overline{Z}=
    \mathfrak{S}_{\overline{X},\overline{Y},\overline{Z}}T(\widehat{ {R}}(\overline{X},\overline{Y}),\overline{Z}).
 \end{equation}
Hence, the result follows from (\ref{eq.2}) and (\ref{eq.3}).
\end{proof}

\begin{lem}\label{ricc.hv} For a $\pi$-tensor field
$\omega$ of type $(1,1)$ on a Finsler manifold $(M,L)$, we have
\begin{eqnarray}
  (\stackrel{v}{\nabla}\stackrel{h}{\nabla}\omega)(\overline{X},\overline{Y},\overline{Z})&-&
(\stackrel{h}{\nabla}\stackrel{v}{\nabla}\omega)(\overline{Y},\overline{X},\overline{Z})=
-( {P}(\overline{X},\overline{Y})\omega)(\overline{Z})\\
&&+(\stackrel{v}{\nabla}\omega)(\widehat{
{P}}(\overline{X},\overline{Y}),\overline{Z})+
(\stackrel{h}{\nabla}\omega)(T(\overline{Y},\overline{X}),\overline{Z}).
\end{eqnarray}
In particular,  for a scalar function $f(x,y)$, we have
$$\stackrel{v}{\nabla}\stackrel{h}{\nabla}f=\stackrel{h}{\nabla}\stackrel{v}{\nabla}f.$$
\end{lem}

\bigskip

\noindent\textit{\textbf{Proof of Theorem A}}: Let $(M,L)$ be  a concircularly flat Finsler manifold,
then the $(v)h$-torsion tensor $\widehat{ {R}}$ satisfies Equation
(\ref{eq.1}). As a consequence of Lemma \ref{ell}, (\ref{eq.1}) reduces to
\begin{equation}\label{eqn. 2.7}
H=kL^2\phi
\end{equation}

If $k$ is constant, then the result follows from (\ref{eqn. 2.7})
and Definition \ref{sca.}. Now, we will show that
$\stackrel{v}\nabla k=\stackrel{h}\nabla k=0$.
\par We have \cite{r96}
\begin{eqnarray*}
 &&(\nabla_{\gamma\overline{X}} {R})(\overline{Y},\overline{Z},\overline{W})
   + (\nabla_{\beta\overline{Y}} {P})(\overline{Z},\overline{X},\overline{W})-
   (\nabla_{\beta
   \overline{Z}} {P})(\overline{Y},\overline{X},\overline{W})\nonumber\\
&&-
 {P}(\overline{Z},\widehat{ {P}}(\overline{Y},\overline{X}))\overline{W}
+ {R}(T(\overline{X},\overline{Y}),\overline{Z})\overline{W}-
 {S}(\widehat{ {R}}(\overline{Y},\overline{Z}),\overline{X})\overline{W}\nonumber\\
&&+  {P}(\overline{Y}, \widehat{
{P}}(\overline{Z},\overline{X}))\overline{W} -
{R}(T(\overline{X},\overline{Z}),\overline{Y})\overline{W}=0.
\end{eqnarray*}
Setting $\overline{W}=\overline{\eta}$ into the above relation,
noting that $K \circ \gamma = id_{\cp}$,   $K\circ \beta=0$ and
$\widehat{S}=0$, it follows that
\begin{eqnarray*}
 &&(\nabla_{\gamma\overline{X}}\widehat{ {R}})(\overline{Y},\overline{Z})
 -{ {R}}(\overline{Y},\overline{Z})\overline{X}
   + (\nabla_{\beta\overline{Y}}\widehat{ {P}})(\overline{Z},\overline{X})-
   (\nabla_{\beta
   \overline{Z}}\widehat{ {P}})(\overline{Y},\overline{X})\nonumber\\
&&- \widehat{ {P}}(\overline{Z},\widehat{
{P}}(\overline{Y},\overline{X})) + \widehat{ {P}}(\overline{Y},
\widehat{ {P}}(\overline{Z},\overline{X}))+\widehat{
{R}}(T(\overline{X},\overline{Y}),\overline{Z}) -\widehat{
{R}}(T(\overline{X},\overline{Z}),\overline{Y})=0.
\end{eqnarray*}
Applying the  cyclic sum
$\mathfrak{S}_{\overline{X},\overline{Y},\overline{Z}}$ on the above
equation, taking  Lemma \ref{lem.2} into account, we get
\begin{equation}\label{eq.8}
    \mathfrak{S}_{\overline{X},\overline{Y},\overline{Z}}(
    \nabla_{\gamma\overline{X}}\widehat{ {R}})(\overline{Y},\overline{Z})=0.
\end{equation}
Substituting (\ref{eq.1}) into (\ref{eq.8}), using
$(\nabla_{\gamma\overline{X}}\ell)(\overline{Y})=L^{-1}\hbar(\overline{X},\overline{Y})$
(Lemma \ref{ell}(\textbf{b})), we have
\begin{eqnarray*}
   &&
    L(\nabla_{\gamma\overline{Z}}k)\set{(\ell(\overline{X})\overline{Y}-\ell(\overline{Y})\overline{X})} +
    L(\nabla_{\gamma\overline{Y}}k)\set{(\ell(\overline{Z})\overline{X}-\ell(\overline{X})\overline{Z})} \\
    &&+L(\nabla_{\gamma\overline{X}}k)\set{(\ell(\overline{Y})\overline{Z}-\ell(\overline{Z})\overline{Y})}
    +k\ell(\overline{Z})\set{(\ell(\overline{X})\overline{Y}-\ell(\overline{Y})\overline{X})}\\
   &&+k\ell(\overline{Y})\set{(\ell(\overline{Z})\overline{X}-\ell(\overline{X})\overline{Z})}
   +k\ell(\overline{X})\set{(\ell(\overline{Y})\overline{Z}-\ell(\overline{Z})\overline{Y})}\\
   &&+kL\set{(\hbar(\overline{X},\overline{Z})\overline{Y}-\hbar(\overline{Y},\overline{Z})\overline{X})}
   +kL\set{(\hbar(\overline{Z},\overline{Y})\overline{X}-\hbar(\overline{X},\overline{Y})\overline{Z})}\\
   &&+kL\set{(\hbar(\overline{Y},\overline{X})\overline{Z}-\hbar(\overline{Z},\overline{X})\overline{Y})}=0.
\end{eqnarray*}
Setting $\overline{Z}=\overline{\eta}$ into the above relation,
noting that $i_{\overline{\eta}}\ell=L$,
$i_{\overline{\eta}}\hbar=0$ (Lemma \ref{ell}(\textbf{c}))
 and $\nabla_{\gamma\overline{\eta}}k=0$, we
 conclude that
 \begin{equation}\label{eq.2a}
   L^{2}\set{\phi(\overline{Y})\,\nabla_{\gamma\overline{X}}k
   -\phi(\overline{X})\,\nabla_{\gamma\overline{Y}}k}=0.
 \end{equation}
Taking the trace of both sides of (\ref{eq.2a}) with respect to
$\overline{Y}$, noting that $Tr(\phi)=n-1$ \cite{r86}, it follows
that
\begin{equation*}
(n-2)\nabla_{\gamma \overline{X}}k=0.
\end{equation*}
Consequently, as $n\geq3$,
\begin{equation}\label{eq.9}
\stackrel{v}\nabla k=0.
\end{equation}
\par
Now, From  (\ref{eq.1}) and the fact that the
$(v)hv$-torsion $\widehat{ {P}}$ is symmetric \cite{r96}, we get
\begin{eqnarray}
  \mathfrak{S}_{\overline{X},\overline{Y},\overline{Z}}\widehat{ {P}}(\widehat{ {R}}(\overline{X},\overline{Y}),
\overline{Z}) &=&0. \label{eq.5}
\end{eqnarray}
On the other hand, we have \cite{r96}
\begin{equation*}
   \mathfrak{S}_{\overline{X},\overline{Y},\overline{Z}}\,
\{(\nabla_{\beta \overline{X}} {R})(\overline{Y},
\overline{Z},\overline{W})+ {P}(\widehat{
{R}}(\overline{X},\overline{Y}), \overline{Z})\overline{W}\}=0.
\end{equation*}
From which, together with (\ref{eq.5}), it follows that
\begin{equation}\label{eq.6}
   \mathfrak{S}_{\overline{X},\overline{Y},\overline{Z}}\,
(\nabla_{\beta \overline{X}}\widehat{ {R}})(\overline{Y},
\overline{Z})=0.
\end{equation}
Again from (\ref{eq.1}), noting that $\nabla_{\beta
\overline{X}}\ell=0$ (Lemma \ref{ell}(\textbf{b})), (\ref{eq.6})
reads
\begin{eqnarray*}
  && L(\nabla_{\beta
\overline{X}}k)\set{\ell(\overline{Y})\overline{Z}-\ell(\overline{Z})\overline{Y}}
+  L(\nabla_{\beta
\overline{Y}}k)\set{\ell(\overline{Z})\overline{X}-\ell(\overline{X})\overline{Z}} \\
   &&+L(\nabla_{\beta
\overline{Z}}k)\set{\ell(\overline{X})\overline{Y}-\ell(\overline{Y})\overline{X}}=0.
\end{eqnarray*}
Setting $\overline{Z}=\overline{\eta}$ into the above equation,
noting that $\ell(\overline{\eta})=L$  (Lemma
\ref{ell}(\textbf{c})), we obtain
\begin{eqnarray*}
  && L(\nabla_{\beta
\overline{X}}k)\set{\ell(\overline{Y})\overline{\eta}-L\overline{Y}}
+  L(\nabla_{\beta
\overline{Y}}k)\set{L\overline{X}-\ell(\overline{X})\overline{\eta}} \\
   &&+L(\nabla_{\beta
\overline{\eta}}k)\set{\ell(\overline{X})\overline{Y}-\ell(\overline{Y})\overline{X}}=0.
\end{eqnarray*}
Taking the trace of both sides with respect to $\overline{Y}$, it
follows that
\begin{equation}\label{eq.7}
\nabla_{\beta \overline{X}}k=L^{-1}(\nabla_{\beta
\overline{\eta}}k)\ell(\overline{X}).
\end{equation}
Applying the $v$-covariant derivative with respect to $\overline{Y}$
on both sides of (\ref{eq.7}), yields
\begin{equation*}
  \ell(\overline{Y})
  \nabla_{\beta\overline{X}}k+L(\stackrel{v}\nabla\stackrel{h}\nabla
  k)(\overline{X},\overline{Y})=
  L^{-1}\hbar(\overline{X},\overline{Y})(\nabla_{\beta\overline{\eta}}k)+\ell(\overline{X})
  (\stackrel{v}\nabla\stackrel{h}\nabla k)(\overline{\eta},\overline{Y}).
\end{equation*}
Since, $\stackrel{v}\nabla\stackrel{h}\nabla
  k=\stackrel{h}\nabla\stackrel{v}\nabla
  k=0$, by Lemma \ref{ricc.hv} and (\ref{eq.9}), the above relation reduces to
\begin{equation*}
  \ell(\overline{Y})
  \nabla_{\beta\overline{X}}k=
  L^{-1}\hbar(\overline{X},\overline{Y})(\nabla_{\beta\overline{\eta}}k).
\end{equation*}
Setting $\overline{Y}=\overline{\eta}$ into the above equation,
taking Lemma \ref{ell} into account, it follows that
$\nabla_{\beta\overline{X}}k=0$. Consequently,
\begin{equation}\label{eq.11}
\stackrel{h}\nabla k=0.
\end{equation}
\par Now, Equations  (\ref{eq.9}) and (\ref{eq.11}) imply  that
$\nabla k=0$. Hence, $k$ is constant and the theorem is proved.
${\qquad\qquad\quad\qquad\qquad\qquad\qquad\qquad\quad\qquad\qquad\qquad\qquad\qquad}\square$


\Section{Concircularly Recurrent Finsler Manifold}

We first introduce the following two  special Finsler spaces which
will be the object of our study in this section.
\begin{defn}\label{def.4} A Finsler manifold $(M,L)$ of dimension $n\geq3$  is called
 $ {R}^{h}$-recurrent if its $h$-curvature tensor
 $ {R}$ is horizontally recurrent:
  \begin{equation}\label{rec.}
  \stackrel{h}{\nabla} {R}= \lambda\otimes {R} \ \ \ \text{and} \ \   R\neq0,
  \end{equation}
where $\lambda$ is a scalar $\pi$-form,
 positively homogenous of degree zero in $y$, called the recurrence form.
\par In particular, if \,$\stackrel{h}{\nabla} {R}=0$, then $(M,L)$ is
called $ {R}^{h}$-symmetric.
\end{defn}

\begin{defn}\label{def.5} A Finsler manifold $(M,L)$ of dimension $n\geq3$  is called
 concircularly recurrent if its  concircular curvature tensor
 $ {C}$ is horizontally recurrent:
  \begin{equation}\label{co.rec.}
 \stackrel{h}{\nabla} {C}= \alpha\otimes {C} \ \ \  \text{and} \ \  C\neq0,
 \end{equation}
where $\alpha$ is a scalar $\pi$-form, positively homogenous
 of degree zero in $y$, called the recurrence form.
\par In particular, if \,$\stackrel{h}{\nabla} {C}=0$, then $(M,L)$ is
called concircularly symmetric.
\end{defn}

The following theorem is the second main  result of the present
paper.

\vspace{5pt}

\noindent\textbf{Theorem B.}\emph{ Every $R^h$-recurrent
Finsler manifold is concircularly recurrent with the same recurrence
form.}

\begin{proof} Let $(M,L)$ be an $R^h$-recurrent Finsler
manifold with recurrence form $\lambda$. Then (\ref{rec.}) is
satisfied. Consequently,
\begin{equation}\label{eq.12}
   \stackrel{h}{\nabla} {r}= \lambda\otimes {r}.
\end{equation}
Now, from Definition \ref{def.5},  (\ref{rec.}) and (\ref{eq.12}),
we get
\begin{eqnarray*}
 \stackrel{h}{\nabla} {C} &=& \stackrel{h}{\nabla}\set{ {R}-\frac{r}{n(n-1)}\, {G}}\\
   &=& \stackrel{h}{\nabla} {R}-\frac{1}{n(n-1)}\stackrel{h}{\nabla}r\otimes {G} \\
   &=& \lambda\otimes\set{ {R}-\frac{r}{n(n-1)}\, {G}}\\
   &=&\lambda\otimes {C}.
\end{eqnarray*}
Therefore, $(M,L)$ is  concircularly recurrent with the same
recurrence form $\lambda$.
\end{proof}

\begin{rem}\emph{The converse of the above theorem is not true in general.
However, it has been recently proved to be true in Riemannian geometry \cite{Olszak}} .
\end{rem}

\par For the converse of Theorem B to be true in the Finslerian context, an additional
condition is needed, namely the horizontal integrability condition.
A Finsler manifold is said to be  horizontally integrable if its
horizonal distribution is completely integrable (or, equivalently, $\widehat{R}=0)$.
\vspace{8pt}

Now, we are in a position to announce our third main and most important result.

\vspace{5pt}

\noindent\textbf{Theorem C.}\label{th.5}\emph{ Every horizontally
integrable concircularly recurrent Finsler manifold is  $R^h$-recurrent
with the same recurrence form.}

\bigskip

To prove this theorem we need the following three lemmas.

\begin{lem}\label{lem.3}
For a concircularly recurrent Finsler manifold with recurrence form
$\alpha$, we have
$$ \stackrel{h}{\nabla} {R}= \alpha\otimes {R}+\mu\otimes {G},$$
where $\mu$ is a $\pi$-scalar form defined by
$$\mu:=\frac{1}{n(n-1)}\set{\stackrel{h}\nabla r-r\alpha}. $$
\end{lem}
\begin{proof} Let $(M,L)$ be a concircularly recurrent Finsler manifold with recurrence form
$\alpha$. Then,  by Definitions \ref{def.5} and \ref{def.1}, we have
$$ \stackrel{h}{\nabla} \set{R-\frac{r}{n(n-1)}\,G}= \alpha\otimes
 \set{R-\frac{r}{n(n-1)}\,G}.$$
From which, together with the the fact that $\stackrel{h}\nabla
G=0$, we get
\begin{eqnarray*}
 \stackrel{h}{\nabla}R-\frac{\stackrel{h}{\nabla}r}{n(n-1)}\otimes\,G&=& \alpha\otimes
 \set{R-\frac{r}{n(n-1)}\,G}.
\end{eqnarray*}
Hence, the result follows.
\end{proof}

\begin{lem}\label{lem.4} For a horizontally integrable Finsler manifold, we have
\begin{description}
  \item[(a)]$\mathfrak{S}_{\overline{X},\overline{Y},\overline{Z}} {R}(\overline{X},\overline{Y})\overline{Z}=0.$
  \item[(b)]$\textbf{R}(\overline{X},\overline{Y},\overline{Z},\overline{W})=\textbf{R}(\overline{Z},\overline{W},\overline{X},\overline{Y}).$
  \item[(c)] The horizontal Ricci tensors ${R}ic$ is symmetric.
  \item[(d)] $\mathfrak{S}_{\overline{U},\overline{V};\,\,\overline{W},\overline{X};\,\,\overline{Y},\overline{Z}}\set{({R}
(\overline{U},\overline{V})\textbf{R})(\overline{W},\overline{X},\overline{Y},\overline{Z})}=0$
\footnote{$\mathfrak{S}_{\overline{U},\overline{V};\,\,\overline{W},\overline{X};\,\, \overline{Y},\overline{Z}}$
denotes the cyclic sum over the three pairs of arguments
$\overline{U},\overline{V};\,\, \overline{W},\overline{X}$ and
$\overline{Y},\overline{Z}$.}
  \item[(e)]$ (\stackrel{h}{\nabla}\stackrel{h}{\nabla}\omega)(\overline{Y},\overline{X},\overline{Z})-
(\stackrel{h}{\nabla}\stackrel{h}{\nabla}\omega)(\overline{X},\overline{Y},\overline{Z})=
( {R}(\overline{X},\overline{Y})\omega)(\overline{Z}),$\\
where $\omega$ is a $\pi$-tensor field  of type $(1,1)$.
\end{description}
\end{lem}

\begin{proof} ~\par

\vspace{4pt}
 \noindent\textbf{(a)} Follows from (\ref{eq.2}) and the horizontal integrability condition ($\widehat{{R}}=0$).

\vspace{4pt}
 \noindent\textbf{(b)} Follows from \textbf{(a)} and the two identities \cite{r96}:
\begin{equation}\label{low.1}
      \textbf{R}(\overline{X},\overline{Y},\overline{Z}, \overline{W})=
 -\textbf{R}(\overline{Y},\overline{X},\overline{Z},\overline{W}),
\end{equation}
\vspace{-20pt}
\begin{equation}\label{low.2}
     \textbf{R}(\overline{X},\overline{Y},\overline{Z}, \overline{W})=
 -\textbf{R}(\overline{X},\overline{Y}, \overline{W},\overline{Z}).
\end{equation}

\vspace{4pt}
 \noindent\textbf{(c)} Follows from  \textbf{(b)}.

\vspace{4pt}
 \noindent\textbf{(d)} We have:
\begin{eqnarray*}
  ( {R}(\overline{U},\overline{V})\textbf{R})(\overline{W},\overline{X},\overline{Y},\overline{Z})
    &=& -\textbf{R}( {R}(\overline{U},\overline{V})\overline{W},\overline{X},\overline{Y},\overline{Z})-\textbf{R}(\overline{W}, {R}(\overline{U},\overline{V})\overline{X},\overline{Y},\overline{Z})\\
    && -\textbf{R}(\overline{W},\overline{X}, {R}(\overline{U},\overline{V})\overline{Y},\overline{Z})-\textbf{R}(\overline{W},\overline{X},\overline{Y}, {R}(\overline{U},\overline{V})\overline{Z}) \\
(
{R}(\overline{W},\overline{X})\textbf{R})(\overline{Y},\overline{Z},\overline{U},\overline{V})
&=&-\textbf{R}( {R}(\overline{W},\overline{X})\overline{Y},\overline{Z},\overline{U},\overline{V})-\textbf{R}(\overline{Y}, {R}(\overline{W},\overline{X})\overline{Z},\overline{U},\overline{V})\\
&&-\textbf{R}(\overline{Y},\overline{Z}, {R}(\overline{W},\overline{X})\overline{U},\overline{V})-\textbf{R}(\overline{Y},\overline{Z},\overline{U}, {R}(\overline{W},\overline{X})\overline{V})\\
(
{R}(\overline{Y},\overline{Z})\textbf{R})(\overline{U},\overline{V},\overline{W},\overline{X})
&=&-\textbf{R}(
{R}(\overline{Y},\overline{Z})\overline{U},\overline{V},\overline{W},\overline{X})-\textbf{R}(\overline{U}, {R}(\overline{Y},\overline{Z})\overline{V},\overline{W},\overline{X})\\
&&-\textbf{R}(\overline{U},\overline{V},
{R}(\overline{Y},\overline{Z})\overline{W},\overline{X})-\textbf{R}(\overline{U},\overline{V},\overline{W},
{R}(\overline{Y},\overline{Z})\overline{X}).
\end{eqnarray*}
Adding the above three equations, making use of (\ref{low.1}),
(\ref{low.2}) and \textbf{(b)}, the result follows.

 \vspace{4pt}
 \noindent\textbf{(e)}
One can show that for the Cartan connection, we have:
\begin{eqnarray*}
(\stackrel{h}{\nabla}\stackrel{h}{\nabla}\omega)(\overline{X},\overline{Y},\overline{Z})-
(\stackrel{h}{\nabla}\stackrel{h}{\nabla}\omega)(\overline{Y},\overline{X},\overline{Z})&=&
\omega( {R}(\overline{X},\overline{Y})\overline{Z})- {R}(\overline{X},\overline{Y})\omega(\overline{Z})\\
&&+(\stackrel{v}{\nabla}\omega)(\widehat{
{R}}(\overline{X},\overline{Y}),\overline{Z}).
\end{eqnarray*}
From which, together with the assumption  of horizontal integrability, the
result follows.
\end{proof}

\vspace{5pt}

Lemma \ref{lem.4}(\textbf{d}) and the next lemma are the global Finslerian versions of Walker's
lemmas \cite{Walker}, proved locally in Riemannian geometry.

\begin{lem}\label{lem.5} Let $A$ be a symmetric scalar $\pi$-form
and $B$ a scalar $\pi$-form. If for all $\overline{X},\overline{Y},
\overline{Z}\in \cp,$
\begin{equation}\label{eq.15}
  \mathfrak{S}_{\overline{X},\overline{Y}, \overline{Z}}\set{A(\overline{X},\overline{Y})B( \overline{Z})}=0,
\end{equation}
 then $A=0$ or $B=0$.
\par In particular, for a horizontally integrable non-flat
(non-concircularly flat) Finsler manifold, if one of the following
relations holds
\begin{eqnarray*}
  \mathfrak{S}_{\overline{U},\overline{V};\,\overline{W},\overline{X};\,\overline{Y},\overline{Z}}\set{\omega(\overline{U},\overline{V})\textbf{R}(\overline{W},\overline{X},\overline{Y},\overline{Z})
  }&=&0,\\
\mathfrak{S}_{\overline{U},\overline{V};\,\overline{W},\overline{X};\,\overline{Y},\overline{Z}}\set{\omega(\overline{U},\overline{V})\textbf{C}(\overline{W},\overline{X},\overline{Y},\overline{Z})
  }&=&0,
\end{eqnarray*}
then the scalar $\pi$-form $\omega$ vanishes, where
$$\textbf{C}(\overline{X},\overline{Y},\overline{Z},\overline{W}):=g(C(\overline{X},\overline{Y})\overline{Z},\overline{W}).$$
\end{lem}

\begin{proof} Let $A$ be a symmetric scalar $\pi$-form
and $B$ a scalar $\pi$-form which satisfy Relation (\ref{eq.15}).
If $B$ vanishes, the result follows. If $B$ dose not vanish,
then from (\ref{eq.15}), we have
\begin{eqnarray*}
  3A(\overline{X},\overline{X})B(X) &=& 0  \quad \forall \, \overline{X}\in \cp \\
  \Longrightarrow\qquad\qquad\,\, A(\overline{X},\overline{X})&=& 0  \quad \forall \, \overline{X}\in \cp \\
 \Longrightarrow\,\, A(\overline{X}+\overline{Y},\overline{X}+\overline{Y})&=& 0  \quad \forall \, \overline{X}, \overline{Y}\in \cp \\
\Longrightarrow \qquad\qquad 2A(\overline{X},\overline{Y})&=& 0  \quad
\forall \, \overline{X}, \overline{Y}\in \cp.
\end{eqnarray*} Hence, the scalr $\pi$-form $A$ vanishes.
\par
The second part of this lemma follows from the first part, taking
into account the assumption that
$\textbf{R}\neq0\,$($\textbf{C}\neq0$), together with
Lemma \ref{lem.4}(\textbf{b}).
\end{proof}

\bigskip

\noindent\textit{\textbf{Proof of Theorem C}}\,: Let $(M,L)$ be a
horizontally integrable concircularly recurrent Finsler manifold
with recurrence form $\alpha$.  The proof is achieved in three
steps:

\vspace{3pt}

\noindent \emph{\underline{First step}}: \emph{The
$h$-covariant derivative of the recurrence form $\alpha$ is
symmetric}:

\vspace{5pt}

The concircular recurrence condition  (\ref{co.rec.}) gives
\begin{eqnarray*}
  \stackrel{h}{\nabla}\textbf{C}&=&\alpha\otimes\textbf{ C} \\
  \stackrel{h}{\nabla}\stackrel{h}{\nabla}\textbf{C}&=& (\stackrel{h}{\nabla}\alpha)\otimes \textbf{C}
  +\alpha\otimes \stackrel{h}{\nabla}\textbf{C}\\
   &=&(\stackrel{h}{\nabla}\alpha+\alpha\otimes\alpha)\otimes \textbf{C}.
\end{eqnarray*}
From which, taking into account Lemma \ref{lem.4}, we obtain
\begin{eqnarray}
  ( {R}(\overline{U},\overline{V})\textbf{C})(\overline{W},\overline{X},\overline{Y},\overline{Z})&=&(\stackrel{h}{\nabla}\stackrel{h}{\nabla}\textbf{C})(\overline{V},\overline{U},\overline{W},\overline{X},\overline{Y},\overline{Z})
  -(\stackrel{h}{\nabla}\stackrel{h}{\nabla}\textbf{C})(\overline{U},\overline{V},\overline{W},\overline{X},\overline{Y},\overline{Z}) \nonumber\\
   &=& \set{(\stackrel{h}{\nabla}\alpha)(\overline{V},\overline{U})-(\stackrel{h}{\nabla}\alpha)(\overline{U},\overline{V})}\textbf{C}(\overline{W},\overline{X},\overline{Y},\overline{Z})\nonumber \\
   &=&-(\bar{d}\alpha)(\overline{U},\overline{V})\textbf{C}(\overline{W},\overline{X},\overline{Y},\overline{Z}),\label{eq.16}
\end{eqnarray}
where
 \begin{equation}\label{d.par}
(\bar{d}\alpha)(\overline{U},\overline{V}):=(\stackrel{h}{\nabla}\alpha)(\overline{U},\overline{V})
-(\stackrel{h}{\nabla}\alpha)(\overline{V},\overline{U}).
 \end{equation}

\par
On the other hand, in view of Definition \ref{def.1}, we have
  \begin{equation}\label{eq.C}
 \textbf{C}:={\textbf{R}}-\frac{r}{n(n-1)}\,{\textbf{G}},
   \end{equation}
   where $\textbf{G}$ is the $\pi$-tensor field defined by $$\textbf{G}(\overline{X},\overline{Y},\overline{Z},\overline{W})
   :=g(G(\overline{X},\overline{Y})\overline{Z}, \overline{W}).$$
 Using (\ref{eq.C}) and the identities $R(\overline{U},\overline{V})r=0=R(\overline{U},\overline{V})\textbf{G}$, we get
\begin{equation}\label{eq.17}
   ( {R}(\overline{U},\overline{V})\textbf{C})(\overline{W},\overline{X},\overline{Y},\overline{Z})=( {R}(\overline{U},\overline{V})\textbf{R})(\overline{W},\overline{X},\overline{Y},\overline{Z}.)
\end{equation}
Now, from (\ref{eq.16}) and (\ref{eq.17}), taking Lemma
\ref{lem.4}(\textbf{d}) into account, it follows that
\begin{equation*}
   \bar{d}\alpha(\overline{U},\overline{V})\textbf{C}(\overline{W},\overline{X},\overline{Y},\overline{Z})
   +\bar{d}\alpha(\overline{W},\overline{X})\textbf{C}(\overline{Y},\overline{Z},\overline{U},\overline{V})
   +\bar{d}\alpha(\overline{Y},\overline{Z})\textbf{C}(\overline{U},\overline{V},\overline{W},\overline{X})=0
\end{equation*}
From which, together with Lemma \ref{lem.5}, we conclude that
$\bar{d}\alpha=0$. Hence the result follows from (\ref{d.par}).

\bigskip

\noindent\emph{\underline{Second step}}:
 \emph{$(M,L)$ has the property that $R(\overline{X},\overline{Y})\textbf{R}=0$}:

\vspace{5pt}

We have, by Lemma \ref{lem.3},
\begin{eqnarray*}
  \stackrel{h}{\nabla}\textbf{R}&=& \alpha\otimes \textbf{R}+\mu\otimes \textbf{G} \\
  \stackrel{h}{\nabla}\stackrel{h}{\nabla}\textbf{R}&=& (\stackrel{h}{\nabla}\alpha)\otimes \textbf{R}
  +\alpha\otimes \stackrel{h}{\nabla}\textbf{R}+(\stackrel{h}{\nabla}\mu)\otimes \textbf{G} \\
   &=& (\stackrel{h}{\nabla}\alpha)\otimes \textbf{R}
  +\alpha\otimes (\alpha\otimes \textbf{R}+\mu\otimes \textbf{G})+(\stackrel{h}{\nabla}\mu)\otimes
  \textbf{G}\\
&=& (\stackrel{h}{\nabla}\alpha
  +\alpha\otimes \alpha)\otimes \textbf{R}+(\stackrel{h}{\nabla}\mu+\alpha\otimes\mu)\otimes
  \textbf{G}.
\end{eqnarray*}
The above equation together with Lemma \ref{lem.4}\textbf{(e)} and
(\ref{d.par}) imply that
\begin{eqnarray*}
  ( {R}(\overline{U},\overline{V})\textbf{R})(\overline{W},\overline{X},\overline{Y},\overline{Z})&=&
  -(\bar{d}\alpha)(\overline{U},\overline{V})\textbf{R}(\overline{W},\overline{X},\overline{Y},\overline{Z})\nonumber\\
  && -(\bar{d}\mu+\alpha\wedge\mu)(\overline{U},\overline{V})\textbf{G}(\overline{W},\overline{X},\overline{Y},
  \overline{Z}).
 \end{eqnarray*}
Now, taking into account the fact that $\bar{d}\alpha=0$
(\emph{First step}), the above equation reduces to
\begin{eqnarray}
  ( {R}(\overline{U},\overline{V})\textbf{R})(\overline{W},\overline{X},\overline{Y},\overline{Z})
 &=&-(\bar{d}\mu+\alpha\wedge\mu)(\overline{U},\overline{V})\textbf{G}(\overline{W},\overline{X},\overline{Y},\overline{Z}).\label{eq.18}
\end{eqnarray}
From which, taking Lemma \ref{lem.4} into account, we obtain
\begin{eqnarray*}
   &&(\bar{d}\mu+\alpha\wedge\mu)(\overline{U},\overline{V})\textbf{G}(\overline{W},\overline{X},\overline{Y},\overline{Z})
   +(\bar{d}\mu+\alpha\wedge\mu)(\overline{W},\overline{X})\textbf{G}(\overline{Y},\overline{Z},\overline{U},\overline{V}) \\
&&
+(\bar{d}\mu+\alpha\wedge\mu)(\overline{Y},\overline{Z})\textbf{G}(\overline{U},\overline{V},\overline{W},\overline{X})=0
\end{eqnarray*}
Applying Lemma \ref{lem.5}, the above relation implies that
$$\bar{d}\mu+\alpha\wedge\mu=0.$$
Consequently, in view of (\ref{eq.18}), we conclude that
${R}(\overline{U},\overline{V})\textbf{R}=0$.

\bigskip

\noindent\emph{\underline{Third step}}: \emph{$(M,L)$ is
$R^h$-recurrent  with the same recurrence form $\alpha$}:

\vspace{5pt}

We have, from the second step,
\begin{eqnarray}
  &&\textbf{R}( {R}(\overline{U},\overline{V})\overline{W},\overline{X},\overline{Y},\overline{Z})
  +\textbf{R}(\overline{W}, {R}(\overline{U},\overline{V})\overline{X},\overline{Y},\overline{Z})\nonumber \\
  &&+\textbf{R}(\overline{W},\overline{X}, {R}(\overline{U},\overline{V})\overline{Y},\overline{Z})
  +\textbf{R}(\overline{W},\overline{X},\overline{Y}, {R}(\overline{U},\overline{V})\overline{Z})=0\label{eq.19}.
\end{eqnarray}
Differentiating $h$-covariantly  both sides of the above relation
with respect to $\overline{\xi}$, we get
\begin{eqnarray*}
  &&(\nabla_{\beta \overline{\xi}}\textbf{R})( {R}(\overline{U},\overline{V})\overline{W},\overline{X},\overline{Y},\overline{Z})
  +\textbf{R}((\nabla_{\beta \overline{\xi}} {R})(\overline{U},\overline{V})\overline{W},\overline{X},\overline{Y},\overline{Z})\\
 && +(\nabla_{\beta \overline{\xi}}\textbf{R})(\overline{W}, {R}(\overline{U},\overline{V})\overline{X},\overline{Y},\overline{Z})
 +\textbf{R}(\overline{W},(\nabla_{\beta \overline{\xi}} {R})(\overline{U},\overline{V})\overline{X},\overline{Y},\overline{Z})\\
  &&+(\nabla_{\beta \overline{\xi}}\textbf{R})(\overline{W},\overline{X}, {R}(\overline{U},\overline{V})\overline{Y},\overline{Z})
  +\textbf{R}(\overline{W},\overline{X},(\nabla_{\beta \overline{\xi}} {R})(\overline{U},\overline{V})\overline{Y},\overline{Z})\\
  &&+(\nabla_{\beta \overline{\xi}}\textbf{R})(\overline{W},\overline{X},\overline{Y}, {R}(\overline{U},\overline{V})\overline{Z})
  +\textbf{R}(\overline{W},\overline{X},\overline{Y},(\nabla_{\beta \overline{\xi}} {R})(\overline{U},\overline{V})\overline{Z})=0.
\end{eqnarray*}
Applying Lemma \ref{lem.3}, we find
\begin{eqnarray*}
  &&(\alpha(\overline{\xi})\textbf{R}+\mu(\overline{\xi})\textbf{G})( {R}(\overline{U},\overline{V})\overline{W},\overline{X},\overline{Y},\overline{Z})
  +\textbf{R}((\alpha(\overline{\xi}) {R}+\mu(\overline{\xi}) {G})(\overline{U},\overline{V})\overline{W},\overline{X},\overline{Y},\overline{Z})\\
 && +(\alpha(\overline{\xi})\textbf{R}+\mu(\overline{\xi})\textbf{G})(\overline{W}, {R}(\overline{U},\overline{V})\overline{X},\overline{Y},\overline{Z})
 +\textbf{R}(\overline{W},(\alpha(\overline{\xi}) {R}+\mu(\overline{\xi}) {G})(\overline{U},\overline{V})\overline{X},\overline{Y},\overline{Z})\\
  &&+(\alpha(\overline{\xi})\textbf{R}+\mu(\overline{\xi})\textbf{G})(\overline{W},\overline{X}, {R}(\overline{U},\overline{V})\overline{Y},\overline{Z})
  +\textbf{R}(\overline{W},\overline{X},(\alpha(\overline{\xi}) {R}+\mu(\overline{\xi}) {G})(\overline{U},\overline{V})\overline{Y},\overline{Z})\\
  &&+(\alpha(\overline{\xi})\textbf{R}+\mu(\overline{\xi})\textbf{G})(\overline{W},\overline{X},\overline{Y}, {R}(\overline{U},\overline{V})\overline{Z})
  +\textbf{R}(\overline{W},\overline{X},\overline{Y},(\alpha(\overline{\xi}) {R}+\mu(\overline{\xi}) {G})(\overline{U},\overline{V})\overline{Z})=0.
\end{eqnarray*}
\par Now, let us assume that $\mu\neq0$ at a certain  point of $TM$. At
this point, using (\ref{eq.19}), the above equation reduces to
\begin{eqnarray*}
  &&\textbf{G}( {R}(\overline{U},\overline{V})\overline{W},\overline{X},\overline{Y},\overline{Z})+
  \textbf{R}( {G}(\overline{U},\overline{V})\overline{W},\overline{X},\overline{Y},\overline{Z})\\
  &&+\textbf{G}(\overline{W}, {R}(\overline{U},\overline{V})\overline{X},\overline{Y},\overline{Z})+\textbf{R}(\overline{W}, {G}(\overline{U},\overline{V})\overline{X},\overline{Y},\overline{Z})\\
  &&+\textbf{G}(\overline{W},\overline{X}, {R}(\overline{U},\overline{V})\overline{Y},\overline{Z})+\textbf{R}(\overline{W},\overline{X}, {G}(\overline{U},\overline{V})\overline{Y},\overline{Z})\\
  &&+\textbf{G}(\overline{W},\overline{X},\overline{Y}, {R}(\overline{U},\overline{V})\overline{Z})+\textbf{R}(\overline{W},\overline{X},\overline{Y}, {G}(\overline{U},\overline{V})\overline{Z})=0.
\end{eqnarray*}
Using the definition of $\textbf{G}$ and  $ {G}$, the last equality
takes the form
\begin{eqnarray*}
  &&g(\overline{V},\overline{W})\textbf{R}(\overline{U},\overline{X},\overline{Y},\overline{Z})-g(\overline{U},\overline{W})\textbf{R}(\overline{V},\overline{X},\overline{Y},\overline{Z})\\
  &&+g(\overline{V},\overline{X})\textbf{R}(\overline{W},\overline{U},\overline{Y},\overline{Z})-g(\overline{U},\overline{X})\textbf{R}(\overline{W},\overline{V},\overline{Y},\overline{Z})\\
  &&+g(\overline{V},\overline{Y})\textbf{R}(\overline{W},\overline{X},\overline{U},\overline{Z})-g(\overline{U},\overline{Y})\textbf{R}(\overline{W},\overline{X},\overline{V},\overline{Z})\\
  &&+g(\overline{V},\overline{Z})\textbf{R}(\overline{W},\overline{X},\overline{Y},\overline{U})-g(\overline{U},\overline{Z})\textbf{R}(\overline{W},\overline{X},\overline{Y},\overline{V})=0.
\end{eqnarray*}
Taking the   trace of the above equation with respect to the pair of
arguments $(\overline{V},\overline{W})$, we obtain
\begin{eqnarray*}
  &&(n-2)\textbf{R}(\overline{U},\overline{X},\overline{Y},\overline{Z})+\textbf{R}(\overline{Y},\overline{X},\overline{U},\overline{Z})+\textbf{R}(\overline{Z},\overline{X},\overline{Y},\overline{U})\\
  &&-g(\overline{U},\overline{Y})Ric(\overline{X},\overline{Z})+g(\overline{U},\overline{Z})Ric(\overline{X},\overline{Y})=0.
\end{eqnarray*}
This equation, using Lemma \ref{lem.4}\textbf{(a)}, reduces to
\begin{equation}\label{eq.20}
    (n-1)\textbf{R}(\overline{U},\overline{X},\overline{Y},\overline{Z})-
    g(\overline{U},\overline{Y})Ric(\overline{X},\overline{Z})+g(\overline{U},\overline{Z})Ric(\overline{X},\overline{Y})=0.
\end{equation}
Again, taking the trace of the above equation with respect to the
pair of arguments $\overline{X}$ and $\overline{Z}$, we get
$Ric=\frac{r}{n}\,g$, which when inserted to (\ref{eq.20}), gives
$$R=\frac{r}{n(n-1)}\,G.$$ Hence, the concircular curvature $C$
vanishes, which contradicts  our assumption. Therefore, $\mu=0$ at
every point on $TM$. Consequently, by lemma \ref{lem.3}, $(M,L)$ is
$R^h$-recurrent  with the same recurrence form $\alpha$.


\providecommand{\bysame}{\leavevmode\hbox
to3em{\hrulefill}\thinspace}
\providecommand{\MR}{\relax\ifhmode\unskip\space\fi MR }
\providecommand{\MRhref}[2]{%
  \href{http://www.ams.org/mathscinet-getitem?mr=#1}{#2}
} \providecommand{\href}[2]{#2}

\end{document}